\title{Manipulative waiters with probabilistic intuition}

\author{
Ma\l gorzata Bednarska-Bzd\c ega
\thanks{Department of Discrete Mathematics, Faculty of Mathematics and CS, Adam Mickiewicz University, Pozna\'n, Poland. Email: mbed@amu.edu.pl}
\and Dan Hefetz
\thanks{School of Mathematics, University of Birmingham, Edgbaston, Birmingham B15 2TT, United
Kingdom. Email: danny.hefetz@gmail.com. Research supported by EPSRC grant EP/K033379/1.} \and Michael Krivelevich \thanks{School of Mathematical Sciences, Raymond and Beverly Sackler Faculty of Exact Sciences, Tel Aviv University, 69978, Israel. Email: krivelev@post.tau.ac.il. Research supported in part by USA-Israel BSF Grant 2010115 and by grant 912/12 from the Israel Science Foundation.} \and Tomasz \L uczak
\thanks{Department of Discrete Mathematics, Faculty of Mathematics and CS, Adam Mickiewicz University, Pozna\'n, Poland. Email: tomasz@amu.edu.pl. Research supported by NCN grant 2012/06/A/ST1/00261.} 
}

\documentclass[11pt]{article}
\usepackage{amsmath,amsthm,amssymb,latexsym,color,epsfig,a4,enumerate}

\newcommand{\comp}{\mathcal L}
\newcommand{\cyc}{\mathcal Cyc}

\newcommand{\eps}{\varepsilon}

\renewcommand{\mod}{\text{ mod }}
\newcommand{\cA}{\mathcal{A}}

\newcommand{\cF}{\mathcal{F}}
\newcommand{\MB}{\textrm{MB}}
\newcommand{\WC}{\textrm{WC}}
\newtheorem{theorem}{Theorem} [section]
\newtheorem{lemma}[theorem]{Lemma}
\newtheorem{proposition}[theorem]{Proposition}

\newtheorem{definition}[theorem]{Definition}

\newtheorem{conjecture}[theorem]{Conjecture}
\newtheorem{question}[theorem]{Question}

\begin{document}

\maketitle

\begin{abstract}
For positive integers $n$ and $q$ and a monotone graph property $\cA$, we consider the two player, perfect information game $\WC(n,q,\cA)$, which is defined as follows. The game proceeds in rounds. In each round, the first player, called Waiter, offers the second player, called Client, $q+1$ edges of the complete graph~$K_n$ which have not been offered previously. Client then chooses one of these edges which he keeps and the remaining $q$ edges go back to Waiter. If at the end of the game, the graph which consists of the edges chosen by Client satisfies the property $\cA$, then Waiter is declared the winner; otherwise Client wins the game. In this paper we study such games (also known as Picker-Chooser games) for a variety of natural graph theoretic parameters, such as the size of a largest component or the length of a longest cycle. In particular, we describe a phase transition type phenomenon which occurs when the parameter $q$ is close to $n$ and is reminiscent of phase transition phenomena in random graphs. Namely, we prove that if $q \geq (1 + \varepsilon) n$, then Client can avoid components of order $c \varepsilon^{-2} \ln n$ for some absolute constant $c > 0$, whereas, for $q \leq (1 - \varepsilon) n$, Waiter can force a giant, linearly sized component in Client's graph. In the second part of the paper, we prove that Waiter can force Client's graph to be pancyclic for every $q \leq c n$, where $c > 0$ is an appropriate constant. Note that this behaviour is in stark contrast to the threshold for pancyclicity and Hamiltonicity of random graphs.    
\end{abstract}

\begin{quote}
\end{quote}

\section{Introduction}\label{intro}

The theory of positional games on graphs and hypergraphs goes back to the seminal papers of Hales and Jewett~\cite{HJ} and of Erd\H{o}s and Selfridge~\cite{ES}. It has since become a highly developed area of combinatorics (see the monograph of Beck~\cite{TTT} and the recent monograph~\cite{HKSS}). The most popular and widely studied positional games are the so-called Maker-Breaker games $\MB(n,q,\cA)$. In each round of such games, Maker claims one previously unclaimed edge of the complete graph $K_n$ and then Breaker responds by claiming $q$ previously unclaimed edges. Maker's goal is to build a graph which satisfies the monotone increasing property $\cA$, whereas Breaker aims to prevent Maker from achieving his goal. Since this is a finite, perfect information game with no chance moves and without the possibility of a draw, one of the players must have a winning strategy. It has been observed long ago, that very often (although by no means always) the winner of the game could be predicted using a heuristic known as the \emph{probabilistic intuition}. This intuition suggests that the player who has a higher chance to win the game when both players are playing randomly is also the one who wins the game when both players are playing optimally. More precisely, if the random graph $G\left(n, \left\lceil \binom{n}{2}/(q+1) \right\rceil \right)$ satisfies property $\cA$ with probability tending to 1 as $n$ tends to infinity, then Maker has a winning strategy for $\MB(n,q,\cA)$. If, on the other hand, this probability tends to 0 as $n$ tends to infinity, then $\MB(n,q,\cA)$ is Breaker's win. Natural examples of this fascinating phenomenon are the cases when $\cA$ is the property of being connected~\cite{GSz} and when it is the property of being Hamiltonian~\cite{kriv}.

Another class of well-studied positional games are the so-called Avoider-Enforcer games, in which Enforcer aims to force Avoider to build a graph which satisfies some monotone increasing property. Such games are sometimes referred to as mis\`ere Maker-Breaker games. There are two different sets of rules for Avoider-Enforcer games: \emph{strict rules} under which the number of edges a player claims per round is precisely his bias and \emph{monotone rules} under which the number of edges a player claims per round is at least as large as his bias (for more information on Avoider-Enforcer games see, for instance,~\cite{HKSae, HKSSae}).

In this paper we consider positional games which are closely related to Maker-Breaker and Avoider-Enforcer games; the main difference between these game types is the process of selecting edges. In every round of the \emph{Waiter-Client} game $\WC(n,q,\cA)$, the first player, called Waiter, offers the second player, called Client, $q+1$ previously unclaimed edges of $K_n$. Client then chooses one of these edges which he keeps and the remaining $q$ edges go back to Waiter (if in the final round of the game fewer than $q+1$ unclaimed edges remain, then all of them are claimed by Waiter). The game ends as soon as all edges of $K_n$ have been claimed. Waiter wins $\WC(n,q,\cA)$ if, at the end of the game, the graph consisting of all vertices of $K_n$ and all edges claimed by Client satisfies the monotone increasing property $\cA$; otherwise Client is the winner. We will sometimes refer to $\WC(n,q,\cA)$ as the $(q : 1)$ Waiter-Client game $\cA$ on $E(K_n)$. Waiter-Client games were first defined and studied by Beck under the name of Picker-Chooser (see, e.g.~\cite{becksec}). However, we feel that the names Waiter and Client are more suitable to describe the respective roles of the two players.

The probabilistic intuition turns out to be useful for Waiter-Client games as well. In particular, it is known to hold when $\cA$ is the property of admitting a clique of a given fixed order~\cite{BHL} and when it is the property of having diameter two~\cite{ctcs}.

This article is devoted to the study of Waiter-Client games with respect to several natural \emph{global} graph properties. In the first part of the paper we consider the order of a largest component in Client's graph. Playing a $(q : 1)$ Waiter-Client game on $E(K_n)$, and assuming both Waiter and Client follow their optimal strategies, let $\comp(n,q)$ denote the order of a largest component in Client's graph if Client tries to minimize this quantity and Waiter tries to maximize it. We prove the following phase transition type result.

\begin{theorem} \label{th::largeComponent}
Let $n$ be a sufficiently large integer and let $0 < \varepsilon = \varepsilon(n) < 1$. 
\begin{enumerate}[{\bf(i)}]
\item\label{comp_client}
If $q \geq (1 + \varepsilon) n$, then Client has a strategy to ensure that $\comp(n,q) \leq c \varepsilon^{-2} \ln n$ will hold for some absolute constant $c$. 
\item\label{comp_waiter}
If $q \leq (1 - \varepsilon) n$, then Waiter has a strategy to ensure that $\comp(n,q) \geq \min \{n, 2 \varepsilon n - 2\}$.
\end{enumerate}
\end{theorem} 

Theorem~\ref{th::largeComponent} is a new and remarkable manifestation of the probabilistic intuition. It is well-known that, when $q$ is close to $n$, the size of a largest component in Client's graph when both players play randomly (and thus Client's graph is the random graph $G(n, \lfloor \binom{n}{2}/(q+1) \rfloor)$) undergoes a phase transition (see \cite{ER} but also~\cite{Bol, JLR}). Namely, if $q \geq (1 + \varepsilon) n$, then there exists an absolute constant $c > 0$ such that, asymptotically almost surely (or a.a.s. for brevity), every  component in RandomClient's graph has at most $c \varepsilon^{-2} \ln n$ vertices. On the other hand, if $q \leq (1 - \varepsilon) n$, then a.a.s. there exists a component on at least $(2+o_{\varepsilon}(1)) \varepsilon n$ vertices in RandomClient's graph. By the aforementioned results, for every fixed $\varepsilon > 0$ and every $q$ which is not in the \emph{critical window} $((1 - \varepsilon )n, (1 + \varepsilon) n)$, the size of a largest component in Client's graph when both players play randomly and when both players follow their optimal strategies is of the same order of magnitude. Moreover, the dependency on $\varepsilon$ exhibited by random graphs in the \emph{sub-critical regime} and the \emph{super-critical regime}, matches the bounds stated in Theorem~\ref{th::largeComponent}.       

Note that Maker-Breaker games exhibit an even stronger phase transition type behavior than that of random graphs. Indeed, it was proved in~\cite{mbtrans} that if $q = c n$ for some constant $c < 1$, then Maker can build a component on $\Theta(n)$ vertices, whereas, if $c > 1$, then Breaker can ensure that the order of every component in Maker's graph will be bounded from above by some constant (which depends on $c$). Moreover, it was shown in~\cite{mbtrans} that the width of the critical window is $O(\sqrt{n})$. Our next result shows that, similarly to the case of random graphs, if $q = c n$ for some constant $c < 1$, then Waiter can force Client to build a connected component of order $\Theta(\ln n)$.   

\begin{proposition} \label{conn:sparse}
If $q = cn$ for some constant  $c > 0$, then $\comp(n,q) = \Omega(\ln n)$.  
\end{proposition} 

Next, we consider the connectivity game, that is, the game in which Waiter's goal is to ensure $\comp(n,q) = n$. Similarly to the Avoider-Enforcer connectivity game, played under strict rules~\cite{HKSae}, we determine the exact threshold bias for the Waiter-Client version. 

\begin{theorem} \label{th::connectivity}
For every integer $n \geq 4$, Waiter can force Client to build a connected subgraph of $K_n$ if and only if $q \leq \lfloor n/2 \rfloor - 1$. 
\end{theorem} 

It is interesting to note that, exactly as with strict Avoider-Enforcer games (see Theorem 1.5 in~\cite{HKSae}), as soon as Client has $n-1$ edges, his graph is forced to be connected. On the other hand, for Maker-Breaker games~\cite{GSz}, monotone Avoider-Enforcer games~\cite{HKSSae} and random graphs~\cite{Bol,JLR}, connectivity requires $(1/2 + o(1)) n \ln n$ edges. 

\bigskip 

In random graph theory for a long time the threshold for Hamiltonicity had been suspected to be very close to the connectivity threshold but, not surprisingly, dealing with Hamiltonicity turned out to be much harder than handling connectivity. The question of finding the threshold for the existence of a Hamiltonian path had been raised already in the seminal paper of Erd\H{o}s and R\'enyi in 1960. It took sixteen years to determine this threshold up to a constant factor (see \cite{Posa}), and another seven to find it exactly (see {\cite{AKS, Bham, KSz}). 

An analogous question concerning Maker-Breaker games was raised in 1978 by Chv\'atal and Erd\H{o}s~\cite{CE}. To make it more precise, let us define the \emph{threshold bias} of the game $\MB(n,q,{\mathcal A})$ or $\WC(n,q,{\mathcal A})$, as the smallest integer $q = q(n)$ for which Breaker (respectively Client) has a winning strategy for this game. Let ${\mathcal H} = {\mathcal H}(n)$ denote the family of all Hamiltonian graphs on the vertex set $[n]$. The threshold bias of the Maker-Breaker Hamiltonicity game $\MB(n,q,{\mathcal H})$ was determined up to a constant factor by Beck~\cite{beckham}, and computed up to a factor of $1+o(1)$ by Krivelevich~\cite{kriv} only in 2011. The asymptotic value of this threshold bias, anticipated already by Chv\'atal and Erd\H{o}s, follows the probabilistic intuition -- the threshold bias turns out to be $\left(1+o(1)\right)n/\ln n$ and it corresponds precisely to the Hamiltonicity threshold in the evolution of the random graph $G\left(n, \left\lceil \binom{n}{2}/(q+1) \right\rceil \right)$. Let us also mention that the threshold bias of the Hamiltonicity Avoider-Enforcer game, played under monotone rules, has the same value asymptotically (\cite{HKSSae}, \cite{cover}). 

On the contrary, the probabilistic intuition fails for the Waiter-Client Hamiltonicity game. Our main result is that the threshold bias of the game $\WC(n,q,{\mathcal H})$ is of linear order. In fact we prove an even stronger result, namely, that there exists a small positive constant $c$ such that Waiter can force Client to build a pancyclic graph (i.e., a graph admitting a cycle of length $k$ for every $3 \leq k \leq n$) whenever $q \leq c n$.

\begin{theorem} \label{th::pancyclic}
Let $n$ be a sufficiently large integer and let $q = q(n)$ be an integer. Then, playing a $(q : 1)$ Waiter-Client game on $E(K_n)$, the following hold: 
\begin{enumerate}[{\bf(i)}]
\item 
If $q \geq 1.1 n$, then Client has a strategy to keep his graph acyclic throughout the game. 
\item 
There exists a positive constant $c$ such that, if $q \leq c n$, then Waiter has a strategy to ensure that, at the end of the game, Client's graph will be pancyclic. 
\end{enumerate}
\end{theorem}

It readily follows from Theorems~\ref{th::pancyclic} and~\ref{th::connectivity} that the threshold biases for the Hamiltonicity and pancyclicity Waiter-Client games are of the same order. Thus, in a way, the probabilistic intuition does not fail completely in this case, since in the random graph $G\left(n, \left\lceil \binom{n}{2}/(q+1) \right\rceil \right)$ the thresholds for Hamiltonicity and pancyclicity are of the same order (see \cite{CF} and \cite{L}). In this respect Maker-Breaker games prove quite different: the threshold bias for pancyclicity is close to $\Theta(\sqrt{n})$ (see \cite{pancMB}), whereas the threshold bias for Hamiltonicity is $\left(1+o(1)\right)n/\ln n$.

It readily follows from Theorems~\ref{th::connectivity} and~\ref{th::pancyclic} that the threshold bias of the Waiter-Client connectivity game and the threshold bias of the Waiter-Client Hamiltonicity game are of the same order of magnitude. As noted above, an even stronger connection holds for the Maker-Breaker versions of these games. Namely, the threshold biases of both games are asymptotically equal. Combined with Theorem~\ref{th::connectivity}, our next result shows that this is not the case with the Waiter-Client versions of these games.    

\begin{proposition} \label{prop::smallerThreshold}
Let $n$ be a sufficiently large integer and let $q = q(n) \geq 0.49n$ be an integer. Then, playing a $(q : 1)$ Waiter-Client game on $E(K_n)$, Client can ensure that, at the end of the game, the minimum degree in his graph will be at most one.   
\end{proposition} 

To end this section, let us briefly comment on the techniques we use to prove Theorem \ref{th::pancyclic}. We prove and use various properties of expanders and apply P\'osa's ingenious extension-rotation technique. A key novelty of our proof is Lemma~\ref{exp_cyc}, which asserts that Waiter can force Client to build a spanning subgraph of $K_n$ with good expanding properties, even when playing with a linear bias. This lemma is of independent interest and may have additional applications. For example, one can apply a well-known theorem of Friedman and Pippenger~\cite{fried} to prove that, playing with a linear bias, Waiter can force Client to build a \emph{tree-universal} graph. That is, for every positive integer $d$, there are positive constants $c_1$ and $c_2$ (depending on $d$) such that, if $q \leq c_1 n$, then playing a $(q : 1)$ game on $E(K_n)$, Waiter can force Client to build a graph which admits a copy of every tree with at most $c_2 n$ vertices and maximum degree at most $d$.

\section{Preliminaries}

\noindent For the sake of simplicity and clarity of presentation, we do not make a par\-ti\-cu\-lar effort to optimize the constants obtained in some of our proofs. Most of our results are asymptotic in nature and whenever necessary we assume that the number of vertices $n$ is sufficiently large. Our graph-theoretic notation is standard and follows that of~\cite{West}. In particular, we use the following.

For a graph $G$, let $V(G)$ and $E(G)$ denote its sets of vertices and edges respectively, and let $v(G) = |V(G)|$ and $e(G) = |E(G)|$. For disjoint sets $A,B \subseteq V(G)$, let $E_G(A,B)$ denote the set of edges of $G$ with one endpoint in $A$ and one endpoint in $B$, and let $e_G(A,B) = |E_G(A,B)|$. For a vertex $u \in V(G)$ and a set $B \subseteq V(G)$ we abbreviate $E_G(\{u\}, B)$ under $E_G(u, B)$. For a set $S \subseteq V(G)$, let $G[S]$ denote the subgraph of $G$ which is induced on the set $S$. For sets $A, B \subseteq V(G)$, let $N_G(A, B) = \{v \in B \setminus A : \exists u \in A \textrm{ such that } uv \in E(G)\}$ denote the set of neighbors of $A$ in $B \setminus A$. For a vertex $u \in V(G)$ and a set $B \subseteq V(G)$ we abbreviate $N_G(\{u\}, B)$ under $N_G(u, B)$ and let $d_G(u, B) = |N_G(u, B)|$ denote the \emph{degree} of $u$ in $B$. We abbreviate $N_G(A, V(G))$, $N_G(u, V(G))$ and $d_G(u, V(G))$ under $N_G(A)$, $N_G(u)$ and $d_G(u)$, respectively. Often, when there is no risk of confusion, we omit the subscript $G$ from the notation above. A path in a graph is said to be \emph{non-trivial} if it contains at least one edge. The \emph{circumference} of a graph $G$ is the length of a longest cycle in $G$ (if $G$ has no cycles, then its circumference is set to be infinity).    

Let $X$ be a finite set and let ${\mathcal F}$ be a family of subsets of $X$. The \emph{transversal} family of ${\mathcal F}$ is ${\mathcal F}^* := \{A \subseteq X : A \cap B \neq \emptyset \textrm{ for every } B \in {\mathcal F}\}$. 

Assume that some Waiter-Client game, played on the edge-set of some graph $H = (V,E)$, is in progress (in some of our arguments, we will consider games played on graphs other than $K_n$; a formal definition of such games will be given in the next paragraph). At any given moment during this game, let $G_W = (V, E_W)$ denote the graph spanned by Waiter's edges, let $G_C = (V, E_C)$ denote the graph spanned by Client's edges and let $G_F = (V, E_F)$, where $E_F = E \setminus (E_W \cup E_C)$. The edges of $E_F$ are called \emph{free}. 
  
%%%%%%%%%%%%% WIN criterion %%%%%%%%%%%%%%%%%%%%

\bigskip

We would like to state two known game-theoretic results which will be used repeatedly in this paper. Both are instances of the well-known potential function method which was introduced by Erd\H{o}s and Selfridge~\cite{ES} and further developed by Beck (\cite{beckwf, TTT}). In order to do so, we need the following more general definition of Waiter-Client games, played on a general board $X$. Given a positive integer $q$, a finite set $X$ and a family ${\mathcal F}$ of subsets of $X$, the Waiter-Client game $\WC(X,{\mathcal F},q)$ is defined as follows. In each round, Waiter chooses $q+1$ free elements of $X$ and offers them to Client, who then chooses one of them which he keeps and the remaining $q$ elements are claimed by Waiter. If at some point there are less than $q+1$ free elements of $X$ left, then Waiter claims all of them. Waiter wins $\WC(X,{\mathcal F},q)$ if, by the end of the game, he is able to force Client to claim all elements of some $A \in {\mathcal F}$; otherwise Client wins the game.

As observed by Beck~\cite{TTT}, one can adapt the potential function method of Erd\H{o}s and Selfridge, which is based on the derandomization of the first moment method, to obtain the following result.

\begin{theorem}[implicit in~\cite{TTT}]\label{BESPC} 
Let $q$ be a positive integer, let $X$ be a finite set, let ${\mathcal F}$ be a family of (not necessarily distinct) subsets of $X$ and let $\Phi(\mathcal{F}) = \sum_{A \in \mathcal{F}} (q+1)^{-|A|}$. Then, playing the Waiter-Client game $\WC(X, {\mathcal F}, q)$, Client has a strategy to avoid fully claiming more than $\Phi(\mathcal{F})$ sets in ${\mathcal F}$.  
\end{theorem}

One can prove this theorem by essentially repeating the argument used by Beck to prove his classic winning criterion for Breaker in biased Maker-Breaker games (cf.~\cite{beckwf} or \cite{TTT}); the proof is therefore omitted. 

The following potential type result is a rephrased version of Corollary 1.5 from~\cite{pc}.

\begin{theorem}[\cite{pc}]\label{cpwin}
Let $q$ be a positive integer, let $X$ be a finite set and let ${\mathcal F}$ be a family of subsets of $X$. If
$$
\sum_{A \in {\mathcal F}} 2^{-|A|/(2q-1)} < \frac{1}{2} \,,
$$
then Waiter has a winning strategy for the game $\WC(X, {\mathcal F}^*, q)$.
\end{theorem} 

In other words, Theorem~\ref{cpwin} asserts that Waiter can force Client to claim at least one element from every set $A \in {\mathcal F}$. It allows us to use Beck's well-known \emph{building via blocking} approach.  

Finally, let us remark that offering more board elements in a round cannot help Waiter. Therefore, we will sometimes allow Waiter to do so. Formally, we have the following slightly stronger result. 

\begin{proposition} \label{fast} 
Let $q \leq q'$ be positive integers, let $X$ be a finite set and let ${\mathcal F}$ be a family of subsets of $X$. If Waiter has a winning strategy for the $\WC(X,{\mathcal F},q')$ game, then he also has a strategy to win the game $\WC(X,{\mathcal F},q)$ in at most $\lfloor |X|/(q'+1) \rfloor$ rounds.
\end{proposition}

\begin{proof}
Let $S'$ be a winning strategy for Waiter in $\WC(X,{\mathcal F},q')$. We will use $S'$ to devise a winning strategy $S$ for Waiter in  $\WC(X,{\mathcal F},q)$ as follows. In every round of the game $\WC(X,{\mathcal F},q)$, Waiter consults $S'$. When instructed to offer the elements of some set $A' \subseteq X$, he offers the elements of some arbitrary set $A \subseteq A'$ of size $q+1$ and views the elements of $A' \setminus A$ as if he claimed them. As soon as every element of $X$ was offered by Waiter or viewed by him as claimed, he plays arbitrarily until the real end of the game. Since $S'$ is a winning strategy for Waiter in $\WC(X,{\mathcal F},q')$ and since, in this game, Waiter must offer at least $q' + 1$ elements of $X$ per round, it follows that after at most $\lfloor |X|/(q'+1) \rfloor$ rounds of the game $\WC(X,{\mathcal F},q)$, Client will fully claim some $F \in {\mathcal F}$.        
\end{proof}

The rest of this paper is organized as follows: in Section~\ref{avoid} we prove Proposition~\ref{conn:sparse} and Theorems~\ref{th::largeComponent} and~\ref{th::connectivity}. In Section~\ref{sec::cycles} we prove Theorem~\ref{th::pancyclic} and Proposition~\ref{prop::smallerThreshold}. Finally, in Section~\ref{sec::openprob} we present some open problems and conjectures.

%%==================================================================================

%%%%%%%%%% Phase transition %%%%%%%%%%%%%%%%%%%%%%%%%%%%%%%

%\subsection{Phase transition and connectivity games}

\section {The largest component} \label{avoid}

In this section we look at the size of a largest component Waiter can force in Client's graph. We start this study with a proof of Proposition~\ref{conn:sparse}. 

\begin{proof}[Proof of Proposition~\ref{conn:sparse}]
The main idea behind Waiter's strategy is very simple. First, he forces a large matching in Client's graph. Then, he splits the set of endpoints of the edges in this matching into two roughly equal parts and forces Client to build a large matching on the vertices of one of these parts, thus forcing many pairwise vertex disjoint paths of length $3$ in Client's graph. Similarly, if for some positive integer $t$, Client's graph contains many pairwise vertex disjoint paths of length $2^t-1$, Waiter will split the set of endpoints of these paths into two roughly equal parts and force Client to build a large matching on the vertices of one of these parts, thus forcing many pairwise vertex disjoint paths of length $2^{t+1}-1$ in Client's graph. Simple calculations (which will be detailed below) show that, by repeating this process enough times, Waiter can force a path of length $\Theta(\ln n)$ in Client's graph.       

Let us now describe Waiter's strategy in greater detail. Recall that $q = c n$ for some constant $c > 0$. In light of Proposition~\ref{fast}, we can assume that $c \geq 1$. Let  
$$
t^* = \max \left\{t \in {\mathbb Z} \colon \left(\frac{n}{10q} \right)^{2^t} \cdot 10q \geq q^{2/3} \right\}.
$$ 
A straightforward calculation shows that $2^{t^*} = \Theta(\ln n)$. We will prove by induction on $t$ that, for every integer $0 \leq t \leq t^*$, Waiter can make sure that, at some point during the game, Client's graph will contain pairwise vertex disjoint paths $P_1, \ldots, P_{m_t}$ such that:
\begin{description}
\item [(a)] $m_t \geq \left(\frac{n}{10q} \right)^{2^t} \cdot 10q$; 
\item [(b)] $v(P_i) = 2^t$ for every $1 \leq i \leq m_t$;
\item [(c)] if $1 \leq i < j \leq m_t$, $x$ is an endpoint of $P_i$ and $y$ is an endpoint of $P_j$, then $xy$ is free.
\end{description}
 
Properties (a), (b) and (c) clearly hold for $t=0$ at the beginning of the game. Assuming that at some point during the game, Client's graph contains pairwise vertex disjoint paths $P_1, \ldots, P_{m_t}$ which satisfy properties (a), (b) and (c) for some integer $0 \leq t < t^*$, Waiter will force such paths for $t+1$ as follows. For every $1 \leq i \leq m_t$, let $u_i$ and $v_i$ denote the endpoints of $P_i$. Let $X_t = \{u_1, \ldots, u_{\lfloor m_t/2 \rfloor}\}$ and let $Y_t = \{u_{\lfloor m_t/2 \rfloor + 1}, \ldots, u_{m_t}\}$. Offering only edges of $E(X_t, Y_t)$, Waiter forces Client to build a large matching as follows. Assume that, for some non-negative integer $i$, Client's graph already contains a matching $M_i$ of size $i$ between $X_t$ and $Y_t$. Straightforward calculations show that
\begin{equation}\label{manyedges}
(\lfloor m_t/2 \rfloor - i)(\lceil m_t/2 \rceil - i) - i(q+1) \geq q+1
\end{equation}
holds, provided that
\begin{equation}\label{bigmatch}
i \leq \frac{\lfloor m_t/2 \rfloor^2 - (q+1)}{q+1+m_t} \,.
\end{equation}

By (\ref{manyedges}) and Property (c), we infer that Waiter can offer $q+1$ free edges of $E(X_t, Y_t)$ which are not incident with any vertex of $M_i$. It thus follows by (\ref{bigmatch}) and Property (a) that Waiter can force a matching of size $r$ between $X_t$ and $Y_t$, where 
\begin{equation} \label{eq::largeMatching}
r \geq \frac{\lfloor m_t/2 \rfloor^2 - (q+1)}{q+1+m_t} > \frac{m_t^2/4- m_t-q}{q+1+m_t} = \left(1+o(1)\right) \frac{m_t^2}{4(q+m_t)} \,,
\end{equation}
where in the last equality we used the fact that $m_t \geq q^{2/3}$ holds by the definition of $t^*$. 

Combining~\eqref{eq::largeMatching} and the simple fact that $m_t \leq n \leq q$, we have 
$$ 
r \geq \left(1+o(1)\right) \frac{m_t^2}{8q} \geq \left(1+o(1)\right) \frac{\left(\frac{n}{10q} \right)^{2^{t+1}} \cdot 100q^2}{8q} > \left(\frac{n}{10q} \right)^{2^{t+1}} \cdot 10 q \,,
$$
where in the second inequality we used property (a) for $t$, i.e., the assumption $m_t \geq \left(\frac{n}{10q} \right)^{2^t} \cdot 10q$.

We conclude that Waiter forces at least $\left(\frac{n}{10q} \right)^{2^{t+1}} \cdot 10 q$ pairwise vertex disjoint paths in Client's graph on $2 m_t = 2^{t+1}$ vertices each. This shows that properties (a) and (b) are satisfied for $t+1$. It is evident from Waiter's strategy that the resulting paths satisfy property (c) as well. This concludes our inductive argument. 

Finally, since, as noted above, $2^{t^*} = \Theta(\ln n)$, it follows by the definition of $t^*$ and by properties (a) and (b) for $t^*$, that Waiter can indeed force Client to build a path of length $\Theta(\ln n)$.
\end{proof}  

Our next goal is to prove Theorem~\ref{th::largeComponent}. We consider the upper bound first. Before proving this bound, we will state and prove two simple lemmata.

\begin{lemma} \label{cl::fewTrees}
Let $t_{\ell}(k)$ denote the number of labeled spanning trees of $K_k$ with at most $\ell$ leaves. If $k > 2\ell$, then $t_{\ell}(k) < \frac{(ek)^{2\ell} k!}{(2\ell)!}$.
\end{lemma}

\begin{proof}
Let $T$ be a spanning tree of $K_k$ with at most $\ell$ leaves. It is easy to see that the number of vertices of $T$ of degree at least $3$ is less than $\ell$. Hence, at least $k - 2\ell$ vertices of $T$ are of degree $2$. The label of each such vertex appears exactly once in the Pr\"ufer code of $T$. The claim now follows since the number of appropriate sequences is at most 
$$
\binom{k}{k - 2\ell} \cdot (k-2)(k-3) \cdot \ldots \cdot (2\ell-1) \cdot (2\ell)^{2\ell-2} < \left(\frac{e k}{2 \ell} \right)^{2 \ell} \cdot \frac{(k-2)!}{(2\ell - 2)!} \cdot (2\ell)^{2\ell} < \frac{(ek)^{2\ell} k!}{(2\ell)!} \,.
$$ 
\end{proof}

\begin{lemma} \label{cl::manyPairs}
Let $T = (V,E)$ be a tree on $m\geq 3$ vertices. Then the number of ordered $r$-tuples $(P_1, \ldots, P_r)$ of (not necessarily distinct) non-trivial paths of $T$ whose union forms a subtree of $T$ is at least $(m/4)^{2r}$.  
\end{lemma}

\begin{proof}
For every vertex $w \in V$, let $r_w = \max \{v(C) : C \textrm{ is a component of } T \setminus w\}$. Let $v$ be a vertex for which $r_v = \min \{r_w : w \in V\}$. Suppose  that $C_1$ is a component of $T \setminus v$ for which $v(C_1) > m/2$. Let $u$ denote the unique neighbor of $v$ in $C_1$. It is easy to see that $r_u < r_v$ contrary to our choice of $v$. This contradiction shows that $r_v \leq m/2$. 

We will show that there exists a partition $A \cup B$ of $V \setminus \{v\}$ such that $|A|,|B| \geq m/4$ and $C \subseteq A$ or $C \subseteq B$ for every component $C$ of $T \setminus v$. Indeed, let $C_1, \ldots, C_t$ denote the components of $T \setminus v$, where $|C_1| \leq \ldots \leq |C_t| = r_v$. Let $1 \leq j \leq t$ denote the smallest integer for which $\sum_{i=1}^j |C_i| \geq m/4$. Set $A = \bigcup_{i=1}^j C_i$ and $B = (V \setminus \{v\}) \setminus A$. It is immediate from the definition of $A$ and $B$ that $|A| \geq m/4$ and that $C \subseteq A$ or $C \subseteq B$ for every component $C$ of $T \setminus v$. Moreover, $|B| = m - 1 - |A| \geq \max \{r_v, m - 1 - (\lfloor m/4 \rfloor + r_v)\} \geq m/4$.     

Let $A \cup B$ be a partition of $V \setminus \{v\}$ such that $|A|,|B| \geq m/4$ and $C \subseteq A$ or $C \subseteq B$ for every component $C$ of $T \setminus v$. For every $(a_1, \ldots, a_r) \in A^r$ and $(b_1, \ldots, b_r) \in B^r$ there is an $r$-tuple $(P_1, \ldots, P_r)$ of paths of $T$ such that $a_i$ and $b_i$ are the endpoints of $P_i$ for every $1 \leq i \leq r$. Clearly, the number of such $r$-tuples is $|A|^r \cdot |B|^r \geq (m/4)^{2r}$. Moreover, $\bigcup_{i=1}^r P_i$ is a subtree of $T$ since $v \in \bigcap_{i=1}^r V(P_i)$ for each such $r$-tuple. 
\end{proof}

We can now prove the upper bound on $\comp(n,q)$ in Theorem~\ref{th::largeComponent}.  

\begin{proof}[Proof of Theorem~\ref{th::largeComponent}(\ref{comp_client})]
In order to prove this theorem, we will describe a winning strategy for Client. It readily follows from Lemma~\ref{cl::manyPairs} that, if there is a large component in Client's graph, then, for an appropriately chosen parameter $r$, this component contains many $r$-tuples of paths whose union is a tree. Hence, in order to ensure his graph does not contain a large component, Client will guarantee his graph contains relatively few such $r$-tuples. His strategy for doing so will be based on Theorem~\ref{BESPC}.   

Let $r = r(n)$ be a positive integer to be determined later and let ${\mathcal F}_r$ be the family of all ordered $r$-tuples $(P_1, \ldots, P_r)$ of (not necessarily distinct) non-trivial paths in $K_n$ for which $\bigcup_{i=1}^r P_i$ is a tree. We would like to bound $\Phi(\mathcal{F}_r) = \sum_{A \in {\mathcal F}_r} (q+1)^{-|A|}$ from above. Note that there is some abuse of notation here. Formally, ${\mathcal F}_r$ is actually a multi-family of trees, where every tree $T$ appears several times in ${\mathcal F}_r$, once for every ordered $r$-tuple $(P_1, \ldots, P_r)$ for which $\bigcup_{i=1}^r P_i = T$. We use this notation as we feel it will help the reader remember we are dealing with a multi-family. Let ${\mathcal F}_r^1$ denote the family of all ordered $r$-tuples $(P_1, \ldots, P_r) \in {\mathcal F}_r$ for which $v(\bigcup_{i=1}^r P_i) \leq 6r$ and let ${\mathcal F}_r^2$ denote the family of all ordered $r$-tuples $(P_1, \ldots, P_r) \in {\mathcal F}_r$ for which $v(\bigcup_{i=1}^r P_i) \geq 6r+1$. For $i \in \{1,2\}$, let $\Phi_i = \Phi(\mathcal{F}_r^i)$; then $\Phi(\mathcal{F}_r) = \Phi_1 + \Phi_2$. We will first bound each of these terms separately.

For every $(P_1, \ldots, P_r) \in {\mathcal F}_r^1$, the tree $\bigcup_{i=1}^r P_i$ and the ordered $(2r)$-tuple of endpoints of the $P_i$'s determine $(P_1, \ldots, P_r)$ uniquely. Hence
\begin{equation} \label{eq::phi1}
\Phi_1 \leq \sum_{k = 2}^{6r} \binom{n}{k} k^{k-2} \cdot k^{2r} (q+1)^{-k+1} < n \sum_{k = 2}^{6r} e^k k^{2r} \left(\frac{n}{q} \right)^{k-1} < 6r n e^{6r} \cdot (6r)^{2r} < c_1^r n r^{2r+1} \,,
\end{equation}
where $c_1 > 0$ is some absolute constant.

In order to obtain an effective upper bound on $\Phi_2$ we will be more careful when estimating the number of $r$-tuples whose union is a given tree. If $(P_1, \ldots, P_r) \in {\mathcal F}_r^2$ and $T = \bigcup_{i=1}^r P_i$, then every leaf of $T$ must be an endpoint of some $P_i$. Let $\ell$ denote the number of leaves of $T$ and let $(a_1, \ldots, a_r, b_1, \ldots, b_r)$ be the vector of endpoints, where $a_i$ and $b_i$ are the endpoints of $P_i$ for every $1 \leq i \leq r$. There are $(2r)_{\ell}$ ways to determine the leftmost position of each of these $\ell$ leaves in this vector and $k^{2r-\ell}$ ways to fill in the remaining $2r - \ell$ positions. Hence
\begin{eqnarray} \label{eq::phi2}
\Phi_2 &\leq& \sum_{k = 6r+1}^n \binom{n}{k} \sum_{\ell = 2}^{2r} t_{\ell}(k) \cdot (2r)_\ell \cdot k^{2r-\ell} \cdot (q+1)^{-k+1} \nonumber\\ &\leq& \sum_{k = 6r+1}^n \binom{n}{k}\binom{2r}{r} \sum_{\ell = 2}^{2r} t_{\ell}(k) \cdot \ell! \cdot k^{2r-\ell} \cdot (q+1)^{-k+1} \nonumber\\ &\leq& q \sum_{k = 6r+1}^n \frac{n^k}{k! q^k} \cdot 2^{2r} \cdot k^{2r} \sum_{\ell = 2}^{2r} \frac{(ek)^{2\ell} k!\ell!}{(2\ell)! k^{\ell}} < q \sum_{k = 6r+1}^n \frac{n^k}{q^k} \cdot 4^r e^{4r} k^{2r} \sum_{\ell = 2}^{2r} \frac{k^{\ell}\ell!}{(2\ell)!} \nonumber \\ &<& q \sum_{k = 6r+1}^n \frac{n^k}{q^k} \cdot 4^r e^{4r} k^{2r} \sum_{\ell = 2}^{2r} \left(\frac{k}{\ell} \right)^{\ell}
< q \sum_{k = 6r+1}^n \frac{n^k}{q^k} \cdot 4^r e^{4r} k^{2r} \cdot 2r \cdot \left(\frac{k}{2r} \right)^{2r}  \nonumber \\
&<& (1 + \varepsilon) n \sum_{k = 6r+1}^n \frac{2r \cdot e^{4r} k^{4r}}{r^{2r}(1 + \varepsilon)^k} \,,
\end{eqnarray}
where we used the obvious fact that the number of trees with exactly $\ell$ leaves is not larger than the number of trees with at most $\ell$ leaves in the first inequality, Lemma~\ref{cl::fewTrees} in the third inequality and the fact that $(k/\ell)^{\ell}$ is increasing for $k > 6r$ and $\ell \leq 2r$.
         
A straightforward calculation shows that the function $f(x) = x^{4r} (1 + \varepsilon)^{-x}$ attains its maximum at $x = 4r/\ln(1+\varepsilon)$. Hence, using~\eqref{eq::phi2} and the fact that $\ln(1+x) \sim x$ when $x$ tends to $0$, we infer that
\begin{equation} \label{eq::phi2cont}
\Phi_2 \leq (1 + \varepsilon) n^2 \cdot \frac{2r (4r)^{4r}}{(\ln(1 + \varepsilon))^{4r} r^{2r}} < c_2^r n^2 r^{2r+1}/\varepsilon^{4r} \,,
\end{equation}
where $c_2 > 0$ is some absolute constant.

Combining~\eqref{eq::phi1} and~\eqref{eq::phi2cont}, we conclude that 
\begin{equation}
\Phi(\mathcal{F}_r) < c_3^r n^2 r^{2r+1}/\varepsilon^{4r} \,.
\end{equation}
where $c_3 > 0$ is some absolute constant.

It thus follows by Theorem~\ref{BESPC} that Client has a strategy to ensure that his graph will contain less than $c_3^r n^2 r^{2r+1}/\varepsilon^{4r}$ ordered $r$-tuples from ${\mathcal F}_r$. Suppose that $L$ is a component of $G_C$ of order $s$ and let $T$ be a spanning tree of $L$. It follows from Lemma~\ref{cl::manyPairs} that there are at least $(s/4)^{2r}$ ordered $r$-tuples of non-trivial paths of $T$ whose union is a subtree of $T$. Hence 
\begin{equation} \label{eq::finalBound}
(s/4)^{2r} < c_3^r n^2 r^{2r+1}/\varepsilon^{4r} \,.
\end{equation}

Substituting $r = \lfloor \ln n \rfloor$ in~\eqref{eq::finalBound}, we obtain $s \leq c \varepsilon^{-2} \ln n$, for some absolute constant $c > 0$.
\end{proof}

Part (ii) of Theorem~\ref{th::largeComponent} is an immediate consequence of the following theorem, which will also play a crucial role in the proof of Theorem~\ref{th::connectivity}. 

\begin{theorem} \label{BigCompSup}
For all positive integers $n$ and $q$, Waiter has a strategy to ensure that $\comp(n,q) \geq \min\{n, 2(n-q-1)\}$.  
\end{theorem}

\begin{proof}
Since the assertion of the theorem is trivially true for $q \geq n-2$, we can assume that $q \leq n-3$ and that $n \geq 4$. Moreover, since $\min \{n, 2(n-q-1)\} = n$ whenever $q \leq (n-1)/2 - 1$ and since $\comp(n,q)$ is a non-increasing function of $q$, we may assume that $q \geq (n-1)/2 - 1$. Therefore, in the remainder of this proof we assume that $n \geq 4$ and $(n-1)/2 - 1 \leq q \leq n-3$.

We present a strategy for Waiter and then prove it allows him to ensure that Client's graph will contain a component on $n$ vertices if $q = (n-1)/2-1$ (in particular, we assume that $n$ is odd in this case), or on at least $2(n-q-1)$ vertices if $q \geq n/2 - 1$. The proposed strategy consists of two stages. Roughly speaking, in the first stage Waiter forces Client's graph to be a tree with $n-q-2$ edges, while making sure that there are still many free edges between this tree and the remaining vertices. In the second stage Waiter forces Client to extend his tree by absorbing $\min\{n-q-1, q+1\}$ additional vertices, one at a time.   

In light of Proposition~\ref{fast}, we will sometimes assume that Waiter offers strictly more than $q+1$ edges in a round. At any point during the game, if Waiter is unable to follow the strategy presented below, then he forfeits the game. 

\bigskip

\noindent \textbf{Stage I:} Waiter forces Client to build a tree $T$ which satisfies the following three properties: 
\begin{description}
\item [(a)] $v(T) = n-q-1$. 
\item [(b)] $xy$ is free for every $x, y \in V(K_n) \setminus V(T)$. 
\item [(c)] There exists an ordering $u_1, \ldots, u_{q+1}$ of the vertices of $V(K_n) \setminus V(T)$ such that $d_{G_W}(u_i, V(T)) \leq  i-1$ holds for every $1 \leq i \leq \min\{n-q-1, q+1\}$ and $d_{G_W}(u_i, V(T)) \leq n-q-2$ for every $n-q \leq i \leq q+1$.  
\end{description}
This stage lasts exactly $n-q-2$ rounds and as soon as it is over, Waiter proceeds to Stage II.

\medskip

\noindent \textbf{Stage II:} For every $1 \leq i \leq \min\{n-q-1, q+1\}$, in the $i$th round of this stage Waiter offers Client all the free edges of $E(u_i, V(T) \cup \{u_1, \ldots, u_{i-1}\})$. If $q \geq n/2$, then, additionally, for every $n-q \leq j \leq q+1$, Waiter offers one arbitrary free edge of $E(u_j, V(T) \cup \{u_1, \ldots, u_{i-1}\})$.

\bigskip

It remains to prove that Waiter can indeed follow the proposed strategy without forfeiting the game and that, by doing so, he achieves his goal. We consider each stage separately.

\medskip

\noindent \textbf{Stage I:} We will prove by induction on $i$ the following more general claim: Waiter has a strategy to ensure that the following three properties will hold immediately after the $i$th round for every $1 \leq i \leq n-q-2$:
\begin{description}
\item [(a$'$)] Client's graph is a tree $T_i$ with $i$ edges. 
\item [(b$'$)] $xy$ is free for every $x, y \in V(K_n) \setminus V(T_i)$. 
\item [(c$'$)] There is an ordering $u_1, \ldots, u_{n-i-1}$ of the vertices of $V(K_n) \setminus V(T_i)$ such that 
\begin{enumerate}
\item $d_{G_W}(u_j, V(T_i)) = j-1$ for every $1 \leq j \leq \min\{i+1, q+1\}$.   
\item $d_{G_W}(u_j, V(T_i)) = 0$ for every $i+2 \leq j \leq n-q-1$.
\item $d_{G_W}(u_j, V(T_i)) = i$ for every $n-q \leq j \leq n-i-1$.
\end{enumerate} 
\end{description}
Note that for $i = n-q-2$, Properties (a$'$), (b$'$) and (c$'$) entail Properties (a), (b) and (c) (with $T = T_{n-q-2}$).

\medskip

In the first round Waiter offers edges $x y_1, \ldots, x y_{q+1}$ for arbitrary vertices $x, y_1, \ldots, y_{q+1} \in V(K_n)$. Assume without loss of generality that Client selects $x y_1$. Clearly Properties (a$'$) and (b$'$) are satisfied. Let $z_1, \ldots, z_{n-q-2}$ be an arbitrary ordering of the vertices of $V(K_n) \setminus \{x, y_1, \ldots, y_{q+1}\}$. Taking $(u_1, \ldots, u_{n-2}) = (z_1, y_2, z_2, \ldots, z_{n-q-2}, y_3, \ldots, y_{q+1})$ shows that Property (c$'$) is satisfied as well. This proves our claim for $i=1$. 

Assume then that the claim holds for some $1 \leq i \leq n-q-3$; we will show that it holds for $i+1$ as well. In the $(i+1)$st round, Waiter offers all the free edges of $E(\{u_{n-q}, u_{n-q+1}, \ldots, u_{n-i-1}\} \cup \{u_{i+2}\}, V(T_i)\}$. Since $i \leq n-q-3$, it follows that $i+2 \leq n-q-1$. Hence, by the induction hypothesis, Property (c$'$) for $i$ implies that the total number of edges offered is at least $((n-i-1) - (n-q) + 1) \cdot 1 + 1 \cdot (i + 1) = q + 1$. Client selects one of these edges and it readily follows that Properties (a$'$) and (b$'$) are satisfied immediately after the $(i+1)$st round. By reordering the vertices of $V(K_n) \setminus V(T_{i+1})$ if needed, one can verify that Property (c$'$) is satisfied as well.

\bigskip

\noindent \textbf{Stage II:} It suffices to show that in every round of this stage, Waiter offers at least $q+1$ free edges. Let $t = 0$ if $q = (n-1)/2 - 1$ and $t = 1$ if $q \geq n/2 - 1$. Fix an arbitrary $1 \leq i \leq n - q - 2 + t$ and assume that Waiter did not yet forfeit the game and is about to play the $i$th round of Stage II (note that $q+1 = n-q-2$ for $q = (n-1)/2 - 1$). It follows by the proposed strategy for this stage, that no edge which is incident with $u_i$ was offered by Waiter in the $j$th round for any $1 \leq j < i$. Therefore, we infer by Property (b) that
\begin{equation} \label{eq::freeEdgesb}
|\{u_i u_j \colon 1 \leq j < i \textrm{ and } \; u_i u_j \textrm{ is free}\}| = i-1 
\end{equation}

and by Property (c) that
\begin{equation} \label{eq::freeEdgesc}
|\{u_i w \colon w \in V(T) \textrm{ and } u_i w \textrm{ is free}\}| \geq n - q - 1 - (i - 1) = n - q - i \,.
\end{equation}

Furthermore, because of Property (c), for every $n-q \leq j \leq q+1$ there was at least one free edge in $E(u_j,V(T))$ at the end of Stage I. Similarly, because of Property (b), for every $n-q \leq j \leq q+1$ all edges of $E(u_j, \{u_1, \ldots, u_{i-1}\})$ were free at the end of Stage I. Note that, during the first $i-1$ rounds of Stage II, Waiter has claimed at most $i-1$ edges which are incident to $u_j$. Hence, for every $1 \leq i \leq n - q - 2 + t$ and every $n-q \leq j \leq q+1$, immediately before the $i$th round of Stage II, there is still a free edge in $E(u_j, V(T) \cup \{u_1, \ldots, u_{i-1}\})$. It follows that 
\begin{eqnarray} \label{eq::freeEdgesLargeq}
|\{n-q \leq j \leq q+1 \colon \exists w \in V(T) \cup \{u_1, \ldots, u_{i-1}\} \textrm{ such that } u_j w \textrm{ is free}\}| \geq 2(q+1) - n \,.
\end{eqnarray}

Combining~\eqref{eq::freeEdgesb}, \eqref{eq::freeEdgesc} and~\eqref{eq::freeEdgesLargeq} we conclude that the number of free edges Waiter offers in the $i$th round of Stage II is at least $(i-1) + (n - q - i) + (2(q+1) - n) = q + 1$. 

\medskip 

Since we have shown that Waiter can play according to the proposed strategy without forfeiting the game, it readily follows from the description of the strategy that for every $1 \leq i \leq (n-q-2) + (n-q-2+t)$, immediately after the $i$th round, Client's graph is a tree with $i$ edges. In particular, Client is forced to build a component of order $2(n-q-2) + t + 1 = \min \{n, 2(n - q - 1)\}$ in exactly $2(n-q-2) + t$ rounds.
\end{proof}

We end this section with a (by now, very easy) proof of Theorem~\ref{th::connectivity}.

\begin{proof}[Proof of Theorem~\ref{th::connectivity}] If $q \leq \lfloor n/2 \rfloor - 1$, then it follows by Theorem~\ref{BigCompSup} that Waiter can force Client to build a graph admitting a component on $n$ vertices, i.e. a connected graph. On the other hand, if $q > \lfloor n/2 \rfloor - 1$, then by the end of the game, Client's graph will contain strictly less than $n-1$ edges and will therefore be disconnected.
\end{proof}

\section{Circumference, Hamiltonicity and Pancyclicity} \label{sec::cycles}

The main goal of this section is to prove Theorem~\ref{th::pancyclic}. Starting with Part (i), we will in fact prove several results on the circumference of Client's graph. Consider a $(q:1)$ Waiter-Client game on $E(K_n)$ in which Waiter aims to maximize the circumference of Client's graph and Client tries to minimize it. Assuming both players follow their optimal strategies, denote the length of a longest cycle in Client's graph by $\cyc(n,q)$; if Client's graph is a forest, then we put $\cyc(n,q) = 0$. Our results will demonstrate that $\cyc(n,q)$ exhibits a similar behavior to that of the circumference of the random graph $G\left(n, \left\lfloor \binom{n}{2}/(q+1) \right\rfloor \right)$.  

\begin{theorem} \label{thmcyc}
The following hold for every positive integers $n$ and $q = q(n)$.
\begin{enumerate}[{\bf (i)}]
\item If $q \geq 1.1 n$, then $\cyc(n,q) = 0$.
\item If $q = n + \eta$, where $1 \leq \eta = \eta(n) \leq 0.1 n$, then $\cyc(n,q) \leq 10 n/\eta \cdot \ln \ln(n/\eta)$.
\item If $q = n - \eta$, where $10 n^{3/4} \leq \eta = \eta(n) \leq n-1$, then $\cyc(n,q) \geq \eta/6$.
\end{enumerate}
\end{theorem}

\begin{proof}
In order to prove parts (i) and (ii) of the theorem, we apply the potential method (i.e. Theorem~\ref{BESPC}) to show that Client can avoid cycles of given lengths. For an integer $m \geq 3$ let $\cF_m$ denote the family of edge-sets of all cycles of length at least $m$ in $K_n$. Then 
\begin{eqnarray*} 
\Phi(\cF_m) &=& \sum_{A \in \cF_m} (q+1)^{-|A|} = \sum_{k=m}^{n} \binom{n}{k} \frac{(k-1)!}{2} (q+1)^{-k} < \frac{1}{2} \sum_{k=m}^{\infty} \frac{1}{k} \left(\frac{n}{q}\right)^k \\ &<& \frac{1}{2} \left(\frac{n}{q} \right)^{m-1} \sum_{k=1}^{\infty} \frac{1}{k} \left(\frac{n}{q}\right)^k = \frac{1}{2} \left(\frac{n}{q} \right)^{m-1} \ln \left(\frac{q}{q-n} \right) \,,
\end{eqnarray*}
where the last equality follows from the Taylor expansion $- \ln (1-x) = \sum_{k=1}^{\infty} x^k/k$.

Now, one can easily verify that $\Phi(\cF_3) < 1$ if $q \geq 1.1 n$ and that $\Phi(\cF_m) < 1$ if $q = n + \eta$ for some $1 \leq \eta = \eta(n) \leq 0.1 n$ and $m = \lfloor 10 n/\eta \cdot \ln \ln(n/\eta) \rfloor$. Consequently, Parts (i) and (ii) of the theorem follow from Theorem~\ref{BESPC}.

\bigskip

In order to prove part (iii) of the theorem we will describe a strategy for Waiter. Roughly speaking, Waiter's strategy is as follows. First, he forces client to build a path $P$ of length $\eta/2$. Then, he forces Client to claim many edges with one endpoint in some set $W_1$ which is disjoint from $P$ and the other endpoint among the first $|P|/3$ vertices of $P$. Similarly, he forces Client to claim many edges between the last $|P|/3$ vertices of $P$ and some set $W_2$ which is disjoint from $P$. Finally, Waiter forces Client to claim an edge between $W_1$ and $W_2$ (which is clearly possible if $W_1$ and $W_2$ are large enough). This ensures that Client will build a cycle which  contains the middle $|P|/3$ vertices of $P$. 

Formally, Waiter's strategy consists of the following five simple stages. 

\bigskip

\noindent \textbf{Stage I:} Let $V_1 \cup V_2 \cup V_3$ be a partition of $V(K_n)$ such that $|V_1| = |V_2| = \lceil n^{3/4} \rceil$. Offering only edges of $K_n[V_3]$, Waiter forces Client to build a path $P = (u_1, \ldots, u_{\lceil \eta/2 \rceil})$. 

\medskip

\noindent \textbf{Stage II:} Let $r = \lfloor \sqrt{n} \rfloor$ and let $X_1, \ldots, X_r$ be pairwise disjoint subsets of $V_1$, each of size $\lfloor n^{1/4} \rfloor$. For every $1 \leq i \leq r$, in the $i$th round of this stage, Waiter offers Client $q+1$ arbitrary free edges of $E(X_i, \{u_1, \ldots, u_{\lfloor \eta/6 \rfloor}\})$.

\medskip

\noindent \textbf{Stage III:} Let $Y_1, \ldots, Y_r$ be pairwise disjoint subsets of $V_2$, each of size $\lfloor n^{1/4} \rfloor$. For every $1 \leq i \leq r$, in the $i$th round of this stage, Waiter offers Client $q+1$ arbitrary free edges of $E(Y_i, \{u_{\lceil \eta/3 \rceil}, \ldots, u_{\lceil \eta/2 \rceil}\})$.

\medskip

\noindent \textbf{Stage IV:} Let $W_1 = \{w \in V_1 : d_{G_C}(w, P) > 0\}$ and let $W_2 = \{w \in V_2 : d_{G_C}(w, P) > 0\}$. In the only round in this stage, Waiter offers Client $q+1$ arbitrary free edges of $E(W_1, W_2)$.

\bigskip

\noindent \textbf{Stage V:} Waiter offers all the remaining free edges in an arbitrary order.

\bigskip

It is evident that, if Waiter is able to play according to the proposed strategy, then at the end of the game, Client's graph will contain a cycle of length at least $\eta/2 - 2 \cdot \eta/6 = \eta/6$. Using our assumption that $\eta \geq 10 n^{3/4}$, it is easy to verify that $e(X_i, \{u_1, \ldots, u_{\lfloor \eta/6 \rfloor}\}) \geq q+1$ and $e(Y_i, \{u_{\lceil \eta/3 \rceil}, \ldots, u_{\lceil \eta/2 \rceil}\}) \geq q+1$ hold for every $1 \leq i \leq r$ and that $e(W_1, W_2) \geq q+1$.

\medskip

Hence, Waiter is able to follow Stages II, III and IV of the proposed strategy. As for Stage I, for every $1 \leq i < \eta/2$, in the $i$th round of the game, Waiter forces Client's graph to be a path $P_{i+1} = (u_1, \ldots, u_{i+1})$ such that every edge of $K_n[V_3 \setminus \{u_1, \ldots, u_i\}]$ is free. This is done by offering $q+1$ edges $u_i x_1, \ldots u_i x_{q+1}$ for arbitrary vertices $x_1, \ldots x_{q+1} \in V_3 \setminus \{u_1, \ldots, u_i\}$. These edges exist since $q+1 = n - \eta + 1 \leq |V_3| - i$. Finally, he can trivially follow Stage V of the proposed strategy.   
\end{proof}
 
Note that in the super-critical regime, i.e. when $q < (1 - o(1)) n$, our lower bound on $\cyc(n,q)$, stated in Theorem~\ref{thmcyc}, is of the same order of magnitude as our lower bound on $\comp(n,q)$, stated in Theorem~\ref{BigCompSup}. In particular, if $q$ is such that Waiter can force Client to build a connected graph, then he can also force him to build a cycle of length $\Theta(n)$.

%%%%%%%%%%%%%%% Expanders %%%%%%%%%%%%%%%%%%%%%%%%%%

%\subsection{Expander games}

\bigskip

We now wish to prove part (ii) of Theorem~\ref{th::pancyclic}. Although a formal proof is fairly long and technical, its main ideas are natural and not too complicated. First, Waiter forces Client to build an expander. Since Waiter's bias is linear, this result is quite different from all previous results concerning expander building games and thus requires new ideas. Then, using certain properties of Client's expander and various tools including the DFS algorithm and the building via blocking technique, Waiter forces Client to build a cycle of every short length. Finally, using P\'osa's extension-rotation technique, Waiter forces Client to build a cycle of each remaining length.         

Before describing our proof in greater detail, we discuss the well-known relation between expanders and Hamiltonicity. We begin by recalling some definitions and known results.

\begin{definition} \label{def::expander}
For real numbers $d \geq 1$ and $0 < \varepsilon < 1$, a graph $G$ is called a \emph{$(d,\eps)$-expander} if $|N_{G}(S)| \geq d|S|$ holds for every $S \subseteq V(G)$ such that $|S| \leq \eps |V(G)|$.
\end{definition}

\begin{definition} \label{def::halfExpander}
A bipartite graph $G$ with bipartition $(V_1, V_2)$ is called a \emph{$(d,\eps)$-half-expander} on $(V_1,V_2)$ if $|N_{G}(S)| \geq d|S|$ holds for every $S \subseteq V_1$ such that $|S| \leq \eps |V_1|$.
\end{definition} 

\begin{definition} \label{def::booster}
For a graph $G$, a non-edge $uv \notin E(G)$ is called a \emph{booster} of $G$, if either $G \cup \{uv\}$ is Hamiltonian, or the longest path in $G \cup \{uv\}$ is strictly longer than the longest path in $G$. 
\end{definition} 

The following two lemmata are essentially due to P\'osa~\cite{Posa} (see, e.g., Chapter 8.2 of~\cite{Bol} for more details). 

\begin{lemma} \label{lem::manyEndPoints}
Let $G$ be a $(2,\alpha)$-expander on $n$ vertices and let $u_1 \ldots u_n$ be a Hamilton path in $G$. Let $U = \{z \in V(G) : \textrm{ there exists a Hamilton path in } G \textrm{ between } u_1 \textrm{ and } z\}$, then $|U| \geq \alpha n$.
\end{lemma}

\begin{lemma} \label{lem::boosters}
A connected non-Hamiltonian $(2, \alpha)$-expander on $n$ vertices has at least $\alpha^2 n^2/2$ boosters.
\end{lemma}

We will deduce Theorem~\ref{th::pancyclic}(ii) from a series of lemmata. Our first lemma asserts that, even when playing with a linear bias, Waiter can force Client to claim an edge between any two large disjoint sets. This technical lemma will be helpful while proving that Waiter can force Client to build a graph with good expanding properties.  

\begin{lemma}\label{expwf}
Let $G$ be a complete bipartite graph with bipartition $(U,W)$, where $|U| = n$ and $|W| = \lfloor 5n/6 \rfloor$. Let $d$ and $q = q(n) \leq n/(150d)$ be positive integers. Then, for sufficiently large $n$, playing a $(q : 1)$ Waiter-Client game on $E(G)$, Waiter has a strategy to force Client to claim at least one edge between $A$ and $B$ for any $A \subseteq U$ of size $\lceil n/(7d) \rceil$ and any $B \subseteq W$ of size $\lceil n/7 \rceil$.  
\end{lemma}

\begin{proof}
Let ${\mathcal F}$ be the family of edge-sets of all induced subgraphs $G' \subseteq G$ with $V(G') = A \cup B$ such that $A \subseteq U$ is of size $a = \lceil n/(7d) \rceil$ and $B \subseteq W$ is of size $b = \lceil n/7 \rceil$ and let 
$$
\Psi({\mathcal F}) = \sum_{F \in {\mathcal F}} 2^{-|A|/(2q-1)} \,.
$$
Using this notation, it suffices to prove that Waiter has a winning strategy for the game $\WC(E(G), {\mathcal F}^*, q)$. In light of Theorem~\ref{cpwin}, it is thus enough to verify that $\Psi({\mathcal F}) < 1/2$. Let $m = \lfloor 5n/6 \rfloor$ and let $\delta = 1/150$. Then $q \leq \delta n/d$ and
\begin{eqnarray*} 
\Psi({\mathcal F}) &=&  \binom{n}{a} \binom{m}{b} 2^{- a b/(2q-1)} \leq \left(\frac{e n}{a} \right)^a \left(\frac{e m}{b} \right)^b 2^{- ab/(2q)} \leq \left(20 d \right)^{n/(7d)} \cdot 16^{n/7} \cdot 2^{- \frac{n}{7d} \cdot \frac{n}{7} \cdot \frac{d}{2 \delta n}} \\ 
&=& \left[\left(20 d \right)^{1/(7d)} \cdot 16^{1/7} \cdot 2^{- 1/(98 \delta)} \right]^n \,.
\end{eqnarray*}

For $\delta = 1/150$, it is easy to verify that $1/(98 \delta) > 1/7 \cdot \log_2 16 + 1/(7d) \cdot \log_2 (20d)$. For sufficiently large $n$, it then follows that
$$
\Psi({\mathcal F}) < \frac{1}{2}
$$
as claimed. 
\end{proof}

Our second lemma asserts that, playing on a complete bipartite graph, Waiter can force Client to build a half-expander. 

\begin{lemma}\label{halfexpandPC}
Let $G$ be a complete bipartite graph with bipartition $(V_1,V_2)$, where $|V_1| = n$ and $|V_2| \geq n-1$. Let $d$ and $q = q(n) \leq n/(150d)$ be positive integers. Then, for sufficiently large $n$, playing a $(q : 1)$ Waiter-Client game on $G$, Waiter can force Client to build a $(d,2/(3d))$-half-expander on $(V_1,V_2)$. 
\end{lemma}

\begin{proof}
Let $G$ and $q$ be as in the statement of the lemma. Before a formal analysis, let us present the idea behind Waiter's strategy. First he will force Client to build a spanning subgraph of $G$ which is almost a half-expander. By ``almost'' we mean that some very small subsets of $V_1$ may not expand. Then for every vertex of these exceptional sets Waiter will generate its own $d$ new Client neighbours; Client's graph obtained this way will contain the required half-expander. 

Here is a detailed description of Waiter's strategy. Let $W \subseteq V_2$ be a set of size $\lfloor 5n/6 \rfloor$ and let $G' = G[V_1 \cup W]$. The strategy for Waiter is divided into the following two stages. 

\bigskip

\noindent \textbf{Stage I:} Offering only edges of $G'$, Waiter forces Client to build a graph $H \subseteq G'$ such that $E_H(A,B) \neq \emptyset$ for every $A \subseteq V_1$ of size $a = \lceil n/(7d) \rceil$ and every $B \subseteq W$ of size $b = \lceil n/7 \rceil$. 

\bigskip

\noindent \textbf{Stage II:} At the beginning of this stage, let ${\mathcal S}$ denote the family of all inclusion minimal sets $A \subseteq V_1$ such that $|A| \leq n/(7d)$ and $|N_{H}(A)| < d|A|$. Let $U = \{x_1, \ldots, x_r\}$ denote the union of all sets $A \in {\mathcal S}$. At any point during this stage, let $Y = \{y \in V_2 \setminus W : d_{G_C}(y) = 0\}$. For every $1 \leq i \leq r$ and every $1 \leq j \leq d$, in round $(i-1)d + j$ of this stage, Waiter offers Client $q+1$ arbitrary free edges of $E_G(x_i, Y)$.

\bigskip

It follows by Lemma~\ref{expwf} that Waiter can play according to Stage I of the proposed strategy. Moreover, we claim that $|N_{H}(A)| \geq d|A|$ holds at the end of Stage I for every $A \subseteq V_1$ such that $n/(7 d) \leq |A| \leq 2n/(3d)$. Indeed, suppose for a contradiction that there exists a set $A \subseteq V_1$ such that $n/(7 d) \leq |A| \leq 2n/(3d)$ and yet $|N_{H}(A)| < d|A|$. Then $|W \setminus N_{H}(A)| > 5n/6 - d|A| \geq n/6$ but $E_H(A, W \setminus N_{H}(A)) = \emptyset$ contrary to Stage I of the proposed strategy.

Our next goal is to show that Waiter can play according to Stage II of the proposed strategy as well. We first claim that $r < n/(7d)$. Indeed, otherwise there exist sets $A_1, \ldots, A_t \in {\mathcal S}$ such that $n/(7d) \leq |T| \leq 2n/(7d)$, where $T := \bigcup_{i=1}^t A_i$. However, it follows by the definition of ${\mathcal S}$ and by a simple inductive argument that $|N_{H}(T)| < d|T|$ contrary to our proof that sets of such sizes expand. It now follows that $|Y| \geq |V_2 \setminus W| - d r \geq n/6 - 1 - n/7 > (q+1)d$ holds at any point during Stage II. Hence, Waiter can indeed play according to Stage II of the proposed strategy as claimed. 

In order to complete the proof of our claim, that by the end of the game, Client's graph is a $(d,2/(3d))$-half-expander on $(V_1, V_2)$, it remains to show that small sets expand as well. Let $A \subseteq V_1$ be an arbitrary set of size $1 \leq |A| \leq n/(7d)$ such that $|N_{H}(A)| < d|A|$. Then there exists a partition $B \cup X_1 \cup \ldots \cup X_p$ of $A$ (it is possible that $B = \emptyset$), where $|N_{H}(B)| \geq d|B|$, $p \geq 1$ and $X_i \in {\mathcal S}$ for every $1 \leq i \leq p$ (we simply successively remove inclusion minimal non-expanding subsets of $A$ until we are left with an expanding set). It follows by Stage II of the proposed strategy that $|N_{G_C}(X_i, V_2 \setminus W)| \geq d|X_i|$ holds for every $1 \leq i \leq p$ and that $N_{G_C}(X_i, V_2 \setminus W) \cap N_{G_C}(X_j, V_2 \setminus W) = \emptyset$ holds for every $1 \leq i < j \leq p$. We conclude that $|N_{G_C}(A)| \geq |N_H(B)| +  \sum_{i=1}^p |N_{G_C}(X_i, V_2 \setminus W)| \geq d|B| + \sum_{i=1}^p d|X_i| \geq d|A|$.       
\end{proof}

Our third lemma asserts that every not too large set of vertices of a half-expander can be covered by a matching.

\begin{lemma} \label{lem::expanderMatching}
Let $d \geq 1$ and $r \geq 4$ be integers. Let $X$ and $Y$ be disjoint sets of sizes $|X|, |Y| \in \{r-1, r\}$ and let $G$ be a $(d, 2/(3d))$-half-expander on $(X,Y)$. Then, for every $T \subseteq X$ of size at most $r/(2d)$, there is a matching in $G$ which contains every vertex of $T$.
\end{lemma}

\begin{proof}
Let $T \subseteq X$ be an arbitrary set of size $t \leq r/(2d)$ and let $S \subseteq T$ be non-empty. Since $1 \leq |S| \leq r/(2d) \leq 2(r-1)/(3d)$ and since $G$ is a $(d, 2/(3d))$-half-expander on $(X,Y)$, it follows that $|N_G(S)| \geq |S|$. By Hall's theorem we conclude that the required matching exists. 
\end{proof}

Our fourth lemma asserts that half-expanders admit a long path with certain additional properties.

\begin{lemma} \label{lem::longPath}
Let $d \geq 2$, let $r$ be a sufficiently large integer and let $m = \lceil 6r/5 \rceil$. For $0 \leq i \leq 3$ let $X_i$ be a set of size $|X_i| \in \{r-1, r\}$ and let $G_i$ be a $(d, 2/(3d))$-half-expander on $(X_i, X_{i+1})$, where addition is taken modulo 4. Then there exists a path $P_{m+1} = (v_0 \ldots v_m)$ in $G_0 \cup G_1 \cup G_2 \cup G_3$ such that, for every $0 \leq j \leq m$ and $0 \leq i \leq 3$, the vertex $v_j$ is in the set $X_i$ if and only if $j \equiv i \mod 4$.  
\end{lemma}

\begin{proof}
Let $H = G_0 \cup G_1 \cup G_2 \cup G_3$ and let $\vec{H}$ be obtained from $H$ by orienting an edge $uv$ from $u$ to $v$ if and only if $u \in X_i$ and $v \in X_{i+1}$ for some $0 \leq i \leq 3$ (here, and throughout this proof, $X_{i+1}$ should be read as $X_0$ if $i=3$). Note that, in order to prove the lemma, it suffices to find a directed path of length $m$ in $\vec{H}$ starting at $X_0$. In order to do so, we apply the DFS algorithm to $\vec{H}$ (a similar argument can be found, e.g., in~\cite{KS}).    

For every non-negative integer $t$, we denote by $S^t$ the set of vertices of $\vec{H}$ whose exploration is complete after $t$ steps of the algorithm, by $T^t$ the set of vertices of $\vec{H}$ not visited thus far and put $U^t = V(\vec{H}) \setminus (S^t \cup T^t)$. Note that, for every $t \geq 0$, there are no arcs of $\vec{H}$ from $S^t$ to $T^t$. Moreover, $U^t$ spans a directed path in $\vec{H}$ for every $t \geq 0$. Therefore, it suffices to prove that there exists some $t \geq 0$ for which $|U^t| \geq m+4$.  

Since, $S^0 = \emptyset$, $T^0 = V(\vec{H})$ and, in every step of the algorithm, either $|S^t|$ is increased by 1 or $|T^t|$ is decreased by 1, it follows that there must be a step, say $t_0$, such that $|S^{t_0}| = |T^{t_0}|$. Hence, there exists an index $0 \leq i \leq 3$ such that  
\begin{equation} \label{eq::largeIntersection}
|S^{t_0} \cap X_i| \geq |T^{t_0} \cap X_i| \quad \textrm{and} \quad |S^{t_0} \cap X_{i+1}| \leq |T^{t_0} \cap X_{i+1}| \,.
\end{equation}

Since $G_i$ is a $(d,2/(3d))$-half-expander on $(X_i, X_{i+1})$ and since $S^{t_0} \cap X_i$ has no neighbors in $T^{t_0} \cap X_{i+1}$, we infer that either  
\begin{equation} \label{eq::smallSet}
|S^{t_0} \cap X_i| \leq \frac{2}{3d} \cdot |X_i|
\end{equation}
or
\begin{equation} \label{eq::largeSet}
|S^{t_0} \cap X_i| > \frac{2}{3d} \cdot |X_i| \quad \textrm{and} \quad |T^{t_0} \cap X_{i+1}| \leq |X_{i+1}| - \frac{2}{3} \cdot |X_i| \,.
\end{equation}

Combining~\eqref{eq::largeIntersection} and~\eqref{eq::smallSet} we obtain that 
\begin{equation} \label{eq::longPath1}
|U^{t_0} \cap X_i| \geq |X_i| - |S^{t_0} \cap X_i| - |T^{t_0} \cap X_i| \geq |X_i| - 2 |S^{t_0} \cap X_i| \geq \left(1 - \frac{4}{3d} \right) |X_i| \geq \frac{1}{3} |X_i| \geq \frac{r-1}{3} \,.
\end{equation}

Similarly, combining~\eqref{eq::largeIntersection} and~\eqref{eq::largeSet} we obtain that 
\begin{equation} \label{eq::longPath2}
|U^{t_0} \cap X_{i+1}| \geq |X_{i+1}| - |T^{t_0} \cap X_{i+1}| - |S^{t_0} \cap X_{i+1}| \geq |X_{i+1}| - 2 |T^{t_0} \cap X_{i+1}| \geq \frac{4}{3} \cdot |X_i| - |X_{i+1}| \geq \frac{r-4}{3} \,.
\end{equation}

Since the vertices of $U^{t_0}$ span a directed path in $\vec{H}$ we have that $||U^{t_0} \cap X_i| - |U^{t_0} \cap X_j|| \leq 1$ for every $0 \leq i,j \leq 3$. Consequently, it follows by~\eqref{eq::longPath1} and~\eqref{eq::longPath2} that  
$$
|U^{t_0}| = \sum_{i=0}^3 |U^{t_0} \cap X_i| \geq 4 \left(\frac{r-4}{3} - 1 \right) > \frac{6 r}{5} + 4 \,,
$$  
where the last inequality holds for sufficiently large $r$.
\end{proof}

Our fifth lemma asserts that, playing on $E(K_n)$, Waiter can quickly force Client to build a connected expander which admits a cycle of every short length.

\begin{lemma}\label{exp_cyc}
Let $d \geq 4$ and let $n$ be a sufficiently large integer. If $q \leq n/(1000 d)$, then, playing a $(q : 1)$ Waiter-Client game on $E(K_n)$, Waiter has a strategy to ensure that after at most $200 n d$ rounds, Client's graph $G_C$ will satisfy all of the following properties:
\begin{enumerate}[{\bf (i)}]
\item $G_C$ is a $(d/2-1, 1/(20d))$-expander;
\item $G_C$ contains a cycle of length $k$ for every $3 \leq k \leq \lceil n/6 \rceil$;
\item $G_C$ is connected.
\end{enumerate}
\end{lemma}

\begin{proof}
By Proposition~\ref{fast} we may assume that $q = \lceil n/(1000 d) \rceil$. In order to prove the lemma, we present a strategy for Waiter which is divided into six stages. In the first stage Waiter will ensure Property (i). In the next four stages he will ensure Property (ii); each of these stages is devoted to cycles of a specific remainder modulo 4. Finally, in the last stage Waiter will ensure Property (iii).  

\bigskip

\noindent \textbf{Stage I:} Let $V_1 \cup \ldots \cup V_6$ be an equipartition of $V(K_n)$. In at most $198 n d$ rounds, Waiter forces Client to build $(d, 2/(3d))$-half-expanders $G_1, G_2, G_3, G_4, G_5, G_6, G_7$ on $(V_1, V_2)$, $(V_2, V_3)$, $(V_3, V_4)$, $(V_4, V_1)$, $(V_1, V_5)$, $(V_5, V_6)$, and $(V_6, V_2)$, respectively.    

\bigskip

\noindent \textbf{Stage II:} Let $m = \lceil n/5 \rceil$ and let $P_m = (v_1 \ldots v_m)$ be a path in $G_1 \cup G_2 \cup G_3 \cup G_4$ such that, for every $1 \leq t \leq m$ and $1 \leq r \leq 4$, the vertex $v_t$ is in the set $V_r$ if and only if $t \equiv r \mod 4$. In this stage Waiter offers only edges with both endpoints in $V_1$. For every positive integer $j$ such that $4j + 1 \leq \lceil n/6 \rceil$, in the $j$th round of this stage, Waiter offers all edges of $\{v_{4i+1} v_{4i+4j+1} : 0 \leq i \leq q\}$. 

\bigskip

\noindent \textbf{Stage III:} In this stage Waiter offers only edges with one endpoint in $V_1$ and one endpoint in $V_3$. For every positive integer $j$ such that $4j - 1 \leq \lceil n/6 \rceil$, in the $j$th round of this stage, Waiter offers all edges of $\{v_{4i+1} v_{4i+4j-1} : 0 \leq i \leq q\}$. 

\bigskip

\noindent \textbf{Stage IV:} Let $\{v_{4i+1} x_{4i+1} : 0 \leq i \leq q\}$ be the edges of a matching in $G_5$. In this stage Waiter offers only edges with one endpoint in $V_5$ and one endpoint in $V_3$. For every positive integer $j$ such that $4j \leq \lceil n/6 \rceil$, in the $j$th round of this stage, Waiter offers all edges of $\{x_{4i+1} v_{4i+4j-1} : 0 \leq i \leq q\}$. 

\bigskip

\noindent \textbf{Stage V:} Let $\{x_{4i+1} y_{4i+1} : 0 \leq i \leq q\}$ be the edges of a matching in $G_6$. In this stage Waiter offers only edges with one endpoint in $V_6$ and one endpoint in $V_4$. For every positive integer $j$ such that $4j + 2 \leq \lceil n/6 \rceil$, in the $j$th round of this stage, Waiter offers all edges of $\{y_{4i+1} v_{4i+4j} : 0 \leq i \leq q\}$. 

\bigskip

\noindent \textbf{Stage VI:} This stage is further divided into 5 phases. For $1 \leq i \leq 5$, in the $i$th phase, offering only edges of $K_n[V_{i+1}]$, Waiter forces Client to build a spanning connected subgraph of $K_n[V_{i+1}]$.   

\bigskip

It remains to prove that Waiter can follow the proposed strategy and that, by doing so, he ensures within $200 n d$ rounds that Client's graph will satisfy Properties (i), (ii) and (iii). Starting with the former, it follows from Lemma~\ref{halfexpandPC} that Waiter can play according to Stage I of the proposed strategy. The path $P_m$ needed for Stage II exists by Lemma~\ref{lem::longPath} and the matchings needed for Stages IV and V exist by Lemma~\ref{lem::expanderMatching}. Since, moreover, $4q + k < m$ holds for every $k \leq \lceil n/6 \rceil$, it is straightforward to verify that Waiter can play according to Stages II, III, IV and V of the proposed strategy. Finally, since $q \leq \lfloor |V_i|/2 \rfloor - 1$ holds for every $2 \leq i \leq 6$, it follows by Theorem~\ref{th::connectivity} that Waiter can play according to Stage VI of the proposed strategy.   

Next, we prove that by following the proposed strategy, Waiter ensures that Client's graph will satisfy Properties (i), (ii) and (iii).

Let $G = (V(K_n), \bigcup_{i=1}^7 E(G_i))$, that is, $G$ is Client's graph at the end of Stage I. For every $1 \leq i \leq 7$, let $(V_L(G_i), V_R(G_i))$ denote the bipartition of $G_i$. Note that for every triple of sets $(S_a, S_b, S_c)$ such that $1 \leq a < b < c \leq 7$ and $S_i \subseteq V_L(G_i)$ for every $i \in \{a,b,c\}$, we have $N_{G_a}(S_a) \cap N_{G_b}(S_b) \cap N_{G_c}(S_c) = \emptyset$. Since $G_i$ is a $(d, 2/(3d))$-half-expander for every $1 \leq i \leq 7$, for every $S \subseteq V(G)$ with $|S| \leq n/(20d) < 2/(3d) \cdot \lfloor n/6 \rfloor$ we have   
$$
|N_G(S)| \geq \Big|\bigcup_{i=1}^7 N_{G_i}(S \cap V_L(G_i)) \Big| - |S| \geq \frac{1}{2} \sum_{i=1}^7 |N_{G_i}(S \cap V_L(G_i))| - |S| \geq \Big(\frac{d}{2} - 1 \Big)|S| \,.
$$

Hence, $G$ is $(d/2-1, 1/(20d))$-expander; this proves (i). 

Fix some $3 \leq k \leq \lceil n/6 \rceil$. It is easy to verify that Waiter forced a cycle of length $k$ in Client's graph in Stage II if $k \equiv 1 \mod 4$, in Stage III if $k \equiv 3 \mod 4$, in Stage IV if $k \equiv 0 \mod 4$ and in Stage V if $k \equiv 2 \mod 4$. This proves (ii).  

In Stage VI, Waiter makes sure that $G_C[V_i]$ is connected for every $2 \leq i \leq 6$. Since, moreover, $G_j$ is a $(d, 2/(3d))$-half-expander, for every $1 \leq j \leq 7$, it follows that $G_C$ is connected as well. This proves (iii).  

Finally, we prove that Waiter can achieve his goals quickly. It is evident that Stage I lasts at most 
$$
\frac{7 \cdot e(K_{\lceil n/6 \rceil, \lceil n/6 \rceil})}{q+1} \leq \frac{7 \cdot \lceil n/6 \rceil^2}{n/(1000 d)} \leq 198 n d 
$$
rounds.

In Stages II, III, IV and V, for every $3 \leq k \leq \lceil n/6 \rceil$ Waiter spends exactly one round forcing a cycle of length $k$ in $G_C$. Therefore, Stages II, III, IV and V together last at most $n/6$ rounds.

It follows by Proposition~\ref{fast} and by Theorem~\ref{th::connectivity} that Stage VI lasts at most
$$
\frac{5 \cdot e(K_{\lceil n/6 \rceil})}{\lfloor\lfloor n/6 \rfloor /2 \rfloor} \leq (1+o(1)) \frac{60 n^2}{72 n} < n
$$
rounds.

To summarize, Waiter can force Client to build a graph which satisfies Properties (i), (ii) and (iii) in at most $200 n d$ rounds. 
\end{proof}

Our sixth lemma asserts that, playing on $E(K_n)$, Waiter can force Client to build a Hamiltonian expander which admits a cycle of every short length. The proof of Hamiltonicity is analogous to Beck's proof for Maker-Breaker games~\cite{beckham}. 

\begin{lemma}\label{exp_ham}
There exists a constant $c > 0$ such that for sufficiently large $n$ and $q < c n$, playing a $(q : 1)$ Waiter-Client game on $E(K_n)$, Waiter has a strategy to force Client to build a graph $G_C$ which satisfies all of the following properties:
\begin{enumerate}[{\bf (i)}]
\item $G_C$ is a $(2, 1/120)$-expander;
\item $G_C$ contains a cycle of length $k$ for every $3 \leq k \leq \lceil n/6 \rceil$;
\item $G_C$ is Hamiltonian.
\end{enumerate}
\end{lemma}

\begin{proof}
In order to prove the lemma, we present a strategy for Waiter; it is divided into the following two stages. 

\bigskip

\noindent \textbf{Stage I:} In at most $1200 n$ rounds, Waiter forces Client to build a connected graph which satisfies Properties (i) and (ii).

\bigskip

\noindent \textbf{Stage II:} As long as $G_C$ is not Hamiltonian, in every round of this stage, Waiter offers Client $q+1$ free boosters of his current graph $G_C$. 

\bigskip

Since, by definition, after claiming at most $n$ boosters, Client's graph becomes Hamiltonian, it is evident that if Waiter can follow the proposed strategy, then he can ensure that, at the end of the game, Client's graph will satisfy Properties (i), (ii) and (iii). It thus remains to prove that he can indeed do so.

It follows by Lemma~\ref{exp_cyc}, with $d=6$, that Waiter can play according to Stage I of the proposed strategy. As noted above, Stage II lasts at most $n$ rounds. Therefore, the entire game lasts at most $1201 n$ rounds. Since being an expander is a monotone increasing property, at any point during Stage II, $G_C$ is a $(2, 1/120)$-expander. Therefore, by Lemma~\ref{lem::boosters}, there are at least $n^2/28800$ boosters of $G_C$ in $K_n$. By choosing $c$ to be sufficiently small, we can ensure that $n^2/28800 - 1201 n q > 0$. It follows that, at any point during Stage II, there are enough free boosters for Waiter to offer.     
\end{proof}

Finally, we can complete the proof of the main result of this section.

\begin{proof}[Proof of Theorem~\ref{th::pancyclic}(ii)]
Let $V_1 \cup V_2$ be a partition of $V(K_n)$ such that $|V_2| = n_2 = \lfloor n/7 \rfloor - 1$ and $|V_1| = n_1 = n - n_2$. Let $\tilde{c} < 1/120$ be the constant whose existence in ensured in Lemma~\ref{exp_ham}, applied to $K_{n_2}$, and let $q < \tilde{c} n_2$. In order to prove the theorem, we present a strategy for Waiter. Waiter will force Client to build Hamiltonian expanders on $V_1$ and on $V_2$, which contain cycles of every short length. Then, by offering edges between $V_1$ and $V_2$, Waiter will force long cycles in Client's graph. Here is a formal description of Waiter's strategy, divided into four stages. 

\bigskip

\noindent \textbf{Stage I:} Offering only edges with both endpoints in $V_1$, Waiter forces Client to build a graph $G_1 \subseteq K_n[V_1]$ which satisfies Properties (i), (ii) and (iii) of Lemma~\ref{exp_ham}.  

\bigskip

\noindent \textbf{Stage II:} Offering only edges with both endpoints in $V_2$, Waiter forces Client to build a graph $G_2 \subseteq K_n[V_2]$ which satisfies Properties (i), (ii) and (iii) of Lemma~\ref{exp_ham}.

\bigskip
 
\noindent \textbf{Stage III:} Let $v_1 \ldots v_{n_1}$ be a Hamilton path in $G_1$. In the unique round of this stage, Waiter offers Client $q+1$ free edges of $E(v_{n_1}, V_2)$. 

\bigskip 

\noindent \textbf{Stage IV:} Let $v_1 \ldots v_{n_1} w_1 \ldots w_{n_2}$ be a Hamilton path in $G_C$ and let 
$$
S = \{z \in V_2 : \textrm{ there exists a Hamilton path in } G_2 \textrm{ between } w_1 \textrm{ and } z\} \,.
$$
For every $1 \leq j \leq n_1 - 1$, in the $j$th round of this stage, Waiter offers Client $q+1$ free edges of $E(v_j, S)$.   

\bigskip 

It is evident that Waiter can play according to Stage III of the proposed strategy and it follows by Lemma~\ref{exp_ham} that he can play according to Stages I and II as well. In order to prove that he can play according to Stage IV of the proposed strategy, it suffices to prove that $|S| \geq q+1$. However, since $G_2$ is a $(2, 1/120)$-expander, it follows from Lemma~\ref{lem::manyEndPoints} that $|S| \geq n_2/120 \geq  \tilde{c} n_2 + 1 \geq q+1$, where the second inequality holds for sufficiently large $n$.

Now, fix some $3 \leq k \leq n$. If $3 \leq k \leq \lceil n_1/6 \rceil$, then by Lemma~\ref{exp_ham}, there is a cycle of length $k$ in $G_1$. In order to ensure the existence of long cycles, let $k = n - j + 1$ for some $1 \leq j \leq n_1 - 1$ (note that $n - n_1 + 1 = n_2 + 1 \leq \lceil n_1/6 \rceil$). Let $v_j w$ denote the edge Client claims in the $j$th round of Stage II and let $P_w$ be a path between $w_1$ and $w$ in $G_2$. Then $v_j \ldots v_{n_1} w_1 P_w w v_j$ is a cycle of length $k$ in $G_C$. 
\end{proof}

We end this section with a simple proof of Proposition~\ref{prop::smallerThreshold}. 

\begin{proof}[Proof of Proposition~\ref{prop::smallerThreshold}] Fix some $q \geq 0.49 n$. In order to prove the theorem, we present a strategy for Client; it is divided into the following two simple stages. 

\bigskip

\noindent \textbf{Stage I:} At any point during the game, let $V_0$ denote the set of isolated vertices in $G_C$ and let $U = V(K_n) \setminus V_0$. As long as $|U| < n/4$, Client plays arbitrarily. As soon as $|U| \geq n/4$ first occurs, this stage is over and Client proceeds to Stage II. 

\bigskip

\noindent \textbf{Stage II:} At the end of Stage I, let $x \in V_0$ be a vertex for which $d_{G_W}(x) \geq d_{G_W}(y)$ for every $y \in V_0$. In every round of this stage, if possible, Client claims an arbitrary edge which is not incident to $x$; otherwise, he plays arbitrarily.    

\bigskip

It is evident that Client can follow the proposed strategy. It thus remains to prove that, by doing so, he ensures that $\delta(G_C) \leq 1$ will hold at the end of the game. 

Let $t$ denote the total number of rounds played in Stage I. At the end of Stage I, let $k = |U|$; clearly $k \in \{\lceil n/4 \rceil, \lceil n/4 \rceil + 1\}$ and $t \geq k/2$. At the end of Stage I, we have $e(G_W[U]) \leq \binom{k}{2} - e(G_C) \leq \binom{k}{2} - k/2$. Therefore, at the end of Stage I, the average over $V_0$ of Waiter's degree is
\begin{eqnarray*}   
\frac{1}{n-k} \sum_{u \in V_0} d_{G_W}(u) &\geq& \frac{1}{n-k} \left(tq - \binom{k}{2} + k/2 \right) \geq \frac{k(q-k+2)}{2(n-k)}\\
&\geq& \frac{n(q - n/4)}{8(n - n/4)} = \frac{1}{6} \left(q - \frac{n}{4} \right) > n - 2q \,,
\end{eqnarray*}
where the last inequality holds for $q \geq 0.49 n$. It thus follows by the choice of $x$ that $d_{G_W}(x) > n - 2q$. By the description of the proposed strategy, we conclude that Client will claim at most one edge incident to $x$ and thus $\delta(G_C) \leq d_{G_C}(x) \leq 1$ will hold at the end of the game.  
\end{proof}

%%%%% Concluding remarks %%%%%%%%%%%%%%%%%%%%%%%%%%%

\section{Concluding remarks and open problems}
\label{sec::openprob}

\noindent \textbf{The giant component game}. We have proved that, similarly to the random graph $G(n,p)$, the component structure of Client's graph undergoes a phase transition. Namely, we proved that if at the end of the game, Client's graph contains at most $(1 - \varepsilon)n/2$ edges, where $\varepsilon > 0$ is an arbitrarily small constant, then both players have a strategy to ensure that the size of a largest component in Client's graph will be of order $\ln n$, whereas if Client's graph contains at least $(1 + \varepsilon)n/2$ edges, then Waiter has a strategy to force a giant, linearly sized component in Client's graph. In the sub-critical regime, Client's strategy ensures that every  component in his graph will contain at most $c \varepsilon^{-2} \ln n$ vertices for some constant $c > 0$. This is the same dependency on $\varepsilon$ as in the random graph $G(n, (1 - \varepsilon)/n)$. In the super-critical regime, Waiter's strategy ensures that the largest component in Client's graph will contain at least $2 \varepsilon n - 2$ vertices. This is the same dependency on $\varepsilon$ as in the random graph $G(n, (1 + \varepsilon)/n)$. We believe that the latter bound (as stated in Theorem~\ref{BigCompSup}) is sharp.
 
\begin{conjecture}
For any constant $\varepsilon > 0$ and sufficiently large $n$, if $q = (1 - \varepsilon) n$, then $\comp(n,q) = \min\{n, 2(n-q-1)\}$.  
\end{conjecture}   

It would be interesting to study $\comp(n,q)$ in the critical window, i.e. when $q = (1 + o(1)) n$. We can prove that if $q = n+k$ and $\omega(1) = k = k(n) = o(n)$, then $\comp(n,q) = o(n)$ (this is done by applying Theorem~\ref{BESPC} to the family of labeled non-trivial paths in $K_n$, thus proving $\comp(n,q) \leq 2\sqrt{n/\varepsilon}$), but there is still room for improvement. We would also like to know whether one can obtain a similar result in the super-critical regime, i.e. when $q = n-k$ for some $\omega(1) = k = k(n) = o(n)$. Another challenging task is to determine the width of the critical window.   

\bigskip

\noindent \textbf{A Waiter-Client Lehman type result}. It was proved in Theorem~\ref{th::connectivity} that the threshold bias of the connectivity Waiter-Client game is $\lfloor n/2 \rfloor - 1$. This is \textbf{precisely} the same as the threshold bias of the strict Avoider-Enforcer connectivity game~\cite{HKSae}. However, the arguments used for proving these two results are completely different. In particular, in~\cite{HKSae}, the result follows from a more general Lehman type result (see~\cite{Lehman} for Lehman's Theorem). We were interested whether an analogous result holds for Waiter-Client games as well.   

\begin{question} \label{q::LehmanWC}
Let $q$ be a positive integer and let $G$ be a graph which admits $q+1$ pairwise edge disjoint spanning trees. Playing a $(q : 1)$ Waiter-Client game on $E(G)$, is it true that Waiter can force Client to build a spanning tree?  
\end{question}

For $q=1$, this question was answered affirmatively in~\cite{cdm}. Surprisingly, given a sufficiently large $n=|V(G)|$, for almost all values of $q$, the answer to this question is negative \cite{sylwia}.

\bigskip

\noindent \textbf{Avoiding cycles}. It was proved in Theorem~\ref{thmcyc} that, for $q \geq 1.1 n$, Client can keep his graph acyclic, whereas for $q \leq (1 - \varepsilon) n$, Waiter can force Client to build a cycle. We believe that the latter is asymptotically tight.    

\begin{conjecture}
For any constant $\varepsilon > 0$ and integer $q \geq (1 + \varepsilon)n$, playing a $(q : 1)$ Waiter-Client game on $E(K_n)$, Client has a strategy to keep his graph acyclic.
\end{conjecture}

Note that, similarly to the case of $\comp(n,q)$ we know very little about the behavior of $\cyc(n,q)$ in the critical window, i.e. when $q = (1 + o(1)) n$. 

\bigskip

\textbf{Minimum degree $k$ and $k$-connectivity}. Let $q_{\mathcal A} = q_{\mathcal A}(n)$ denote the threshold bias of the game $\WC(n,q,{\mathcal A})$. We determined precisely the threshold bias of the connectivity game in Theorem~\ref{th::connectivity} and we determined the threshold bias of the Hamiltonicity game up to a multiplicative constant factor in Theorem~\ref{th::pancyclic}. There are other natural graph properties ${\mathcal A}$, which seem simpler than Hamiltonicity, for which we can prove that $q_{\mathcal A} = \Theta(n)$ but cannot determine its asymptotic value. For example, for a positive integer $k$, let ${\mathcal D}_k = {\mathcal D}_k(n)$ denote the property of $n$-vertex graphs having minimum degree at least $k$ and let ${\mathcal C}_k = {\mathcal C}_k(n)$ denote the property of being $k$-vertex-connected. It is not hard to see that $n/(2k) - 3 \leq q_{{\mathcal C}_k} \leq q_{{\mathcal D}_k} \leq n/k$. Indeed, the upper bound follows directly from the simple fact that every graph on $n$ vertices with minimum degree at least $k$ has at least $k n/2$ edges; we can in fact improve it slightly by an argument analogous to the proof of Proposition~\ref{prop::smallerThreshold}. The lower bound can be obtained via the following simple strategy. Let $V_1 \cup \ldots \cup V_k$ be an equipartition of $V(K_n)$. Using Theorem~\ref{th::connectivity}, Waiter first forces Client to build a connected graph on $V_i$ for every $1 \leq i \leq k$. For every $1 \leq i < j \leq k$, let $G_{ij} = (V_i \cup V_j, E(V_i,V_j))$ and let $H_{ij}^1$ and $H_{ij}^2$ be edge disjoint $(q+1)$-regular subgraphs of $G_{ij}$. In an arbitrary order, for every $1 \leq i < j \leq k$, every $x \in V_i$ and every $y \in V_j$, Waiter offers $q+1$ edges of $H_{ij}^1$ incident to $x$ and $q+1$ edges of $H_{ij}^2$ incident to $y$. It is not hard to see that the graph built by Client admits $k$ internally pairwise vertex disjoint paths between any pair of vertices and is thus $k$-connected by Menger's Theorem. We are currently not able to determine $q_{{\mathcal D}_k}$ asymptotically; not even for $k=1$. At present, we are also unable to determine $q_{{\mathcal C}_k}$ asymptotically for any $k \geq 2$. 

Similarly, we do not know the answer to the following two questions (though we suspect it is affirmative):       

\begin{question} \label{q::ConvsDegree}
Is there an integer $k \geq 2$ for which $q_{{\mathcal D}_k}$ is substantially larger than $q_{{\mathcal C}_k}$?   
\end{question}

\begin{question} \label{q::ConvsHam}
Is $q_{{\mathcal C}_2}$ substantially larger than $q_{{\mathcal H}}$? 
\end{question} 

Let us remark that the answers to the analogous questions for Maker-Breaker games are both negative.

\section*{Acknowledgment}

We thank the anonymous referee for helpful comments. The last two authors thank the Mittag-Leffler Institute where a part of the research presented in this paper was conducted.

\bibliographystyle{amsplain}

\end{document}